\documentclass[12pt,oneside]{amsart}

\usepackage{geometry}                
\usepackage{graphicx}
\usepackage{amssymb}

\usepackage{amsmath,amsthm,amscd,amssymb}
\usepackage{latexsym}
\usepackage[colorlinks,citecolor=red,pagebackref,hypertexnames=false]{hyperref}
\usepackage{courier}
\usepackage{listings}

\numberwithin{equation}{section}
\theoremstyle{plain}
\newtheorem{theorem}{Theorem}[section]
\newtheorem{lemma}[theorem]{Lemma}

\theoremstyle{definition}
\newtheorem{algorithm}[theorem]{Algorithm}
\newtheorem{definition}[theorem]{Definition}

\newtheorem{example}[theorem]{Example}

\theoremstyle{remark}

\newtheorem{case[theorem]}{Case}

\def \Z {\mathbb{Z}}

\lstdefinestyle{mystyle}{
    basicstyle=\footnotesize,
    breakatwhitespace=false,         
    breaklines=true,                 
    captionpos=b,                    
    keepspaces=true,                 
    numbers=left,                    
    numbersep=5pt,                  
    showspaces=false,                
    showstringspaces=false,
    showtabs=false,                  
    tabsize=2
}
 
\author{Philipp Birklbauer}
\address{Department of Mathematics, University of Rochester, Rochester, NY}
\email{philipp.birklbauer@rochester.edu}

\title{The Fuglede conjecture holds in $\Z^3_5$ } 
\begin{document}

\begin{abstract}
The Fuglede conjecture states that a set is spectral if and only if it tiles by translation. The conjecture was disproved by T. Tao for dimensions 5 and higher \cite{Tao2004} by giving a counterexample in $\Z_3^5$. We present a computer program that determines that the Fuglede conjecture holds in $\Z_5^3$ by exhausting the search space.

In \cite{Iosevich2017} A. Iosevich, A. Mayeli and J. Pakianathan showed that the Fuglede conjecture holds over prime fields when the dimension does not exceed 2. The question for dimension 3 was previously addressed in \cite{Aten2015} for $p=3$.

In this paper we build upon the results of the work in \cite{Aten2015} to allow a computer to carry out the lengthy computations.
\end{abstract}  

\maketitle

\theoremstyle{plain}
\newtheorem*{function}{Function}

\section{Introduction} 

Fuglede's conjecture was first stated by B. Fuglede in 1974 (\cite{Fuglede}). It proposes that for every subset $E$ of $\mathbb{R}^d$ with positive Lebesque measure, $L^2(E)$ has an orthogonal basis of exponentials (called \emph{$E$ is spectral}) if and only if $E$ tiles $\mathbb{R}^d$ by translation. In subsequent years the conjecture was proven for various different domains.
However, in 2004 T. Tao disproved the conjecture for $d \geq 5$ in \cite{Tao2004}. Furthermore Tao's proof relied on constructing a counterexample, namely a spectral set in $\Z_3^5$ that does not tile.
More counterexamples for different variations of the discrete $\Z_n^d$ case were found subsequently. In 2016 A. Iosevich, A. Mayeli and J. Pakianathan showed that the conjecture does indeed hold in the $\Z_p^2$ case. For the $p$ prime and $d=3$ case it is known that tiling implies spectrum \cite[Theorem 1.1.(g)]{Aten2015}, but the implication in the other direction remains open.

Also, in \cite{Aten2015} the authors proved that the Fuglede conjecture holds in the $d=3$, $p=3$ case. In their paper the authors developed the tools to be able to tackle the problem directly. Unfortunately, they have to point out that their methods are no longer feasible for $p=5$ since the number of spectral sets that need to be tested become too big (even for a computer).

The aim of this paper is to overcome these problems by approaching it from a different perspective, while still relying heavily on the theory developed in \cite{Aten2015}: Instead of generating all possible matrices and testing their rank, we start with just three linearly independent vectors, satisfying certain conditions and see if we can find enough vectors within the span of those three vectors that also fulfill the conditions to then construct a square matrix using those vectors as the rows. This will guarantee that the rank of the resulting matrix is three, as needed for a counterexample. Using this approach we can reduce the search space drastically.

\begin{theorem}
The Fuglede conjecture holds in $\Z_5^3$. That is, a set $V \subset \Z_5^3$ tiles if and only if it is spectral.
\end{theorem}
The remainder of the paper will give the proof of this result with the aid of a computer.

\section{Theoretical results}

\begin{definition}
A vector $v \in \Z_p^{mp}$ is called \emph{balanced} if every element of $\Z_p$ appears exactly $m$ times as a component of $v$.

A matrix $M \in \Z_p^{mp \times mp}$ is called \emph{log-Hadamard} if the difference of any two rows of $M$ is balanced.

A matrix $M \in \Z_p^{mp \times mp}$ is called \emph{special dephased log-Hadamard} if $M$ is a log-Hadamard matrix and the first row and column are all 0 and the second row and column are $(0,1,\dots,p-1,0,\dots,p-1,\dots,0,1,\dots,p-1)$. 
\end{definition}

\begin{theorem}[Proposition 9.1 in \cite{Aten2015}]
\label{theorem:3DFuglede}
Let $p$ be a prime. The following are equivalent: 
\begin{itemize}
\item A subset of $\mathbb{Z}_p^3$ tiles if and only if it is spectral.
\item There does not exist an $mp \times mp$ special dephased log-Hadamard matrix with entries in $\mathbb{Z}_p$ of rank $3$ where $1 < m < p$.
\end{itemize}

\end{theorem}

We immediately get the following lemma:
\begin{lemma}
\label{3indpvec}
Let $1 < m < p$. 
If there exists an $mp \times mp$ special dephased log-Hadamard matrix with entries in $\mathbb{Z}_p$ of rank $3$ then there also exists an $M \in \Z_p^{mp \times mp}$ with the first four rows being $b_0$,$b_1$,$b_2$ and $b_3$ such that:
\begin{enumerate}
\item $M$ is log-Hadamard
\item $b_0$ and the first column are all 0, \label{lemma::3indpvec::b_0=0}
\item $b_1 = (0,1,\dots,p-1,0,\dots,p-1,\dots,0,1,\dots,p-1)$,
\item $b_2 = (0,0,*,*,\ldots,*)$,
\item $b_3$ is linearly independent of $b_1$,$b_2$,
\item all other rows are linear combinations of $b_1, b_2, b_3$, \label{lemma::3indpvec::lincomb}
\item the difference between any two row vectors is balanced. \label{lemma::3indpvec::balanced}
\end{enumerate}
\end{lemma}

\begin{definition}[Definition 5.3. from \cite{Aten2015}]
For $p$ a prime, a $p \times p$ matrix $X$ of nonnegative integer entries whose row and column sums as well as diagonal sums  of the form $\sum_{i=0}^{p-1} X_{i,i+s}$ (for $0 \leq s \leq p-1$) all add up to the same number $m$ is called a \emph{Davey matrix (of weight $m$)}.
\end{definition}

The remarks after \cite[Corollary 5.2.]{Aten2015} give raise to the following lemma:
\begin{lemma}
If $v,w \in \Z_p^{m p}$ are balanced and their difference is balanced as well, then the matrix $X \in \Z^{p \times p}$ with $X_{i,j}= \# \{k : v_k=i, w_k=j, k = 1,\dots,mp\}$ is a Davey matrix. We say $X$ is raised by the pair $(v,w)$.
\end{lemma}

\begin{proof}
$\sum_{k=0}^{p-1} X_{i,k}=\# \{k : v_k=i, k = 1,\dots,mp\}=m$ since $v$ is balanced.
$\sum_{k=0}^{p-1} X_{k,j}=\# \{k : w_k=j, k = 1,\dots,mp\}=m$ since $w$ is balanced.
$\sum_{k=0}^{p-1} X_{k,k+s}=\# \{k : v_k-w_k=s, k = 1,\dots,mp\}=m$ since $v-w$ is balanced.
\end{proof}

\begin{lemma}
\label{Lemma::win_cond}
If the conditions of Lemma \ref{3indpvec} are fulfilled and $b_1$,$b_2$ and $b_3$ are as in the Lemma. Then we know that there are three Davey matrices $D_1$, $D_2$ and $D_3$ (raised by the pairs $(b_1,b_2), (b_1,b_3)$ and $(b_2,b_3)$ respectively) and a set $\Lambda \in \Z_p^3$ of cardinality at least $m p-1$, such that for all $(\lambda_1,\lambda_2,\lambda_3), (\mu_1,\mu_2,\mu_3) \in \Lambda$ we have $(\lambda_1-\mu_1)b_1+(\lambda_2-\mu_2)b_2+(\lambda_3-\mu_3)b_3$ is balanced.
\end{lemma}

\section{Overview of the algorithm}
\subsection{Outline}
The last Lemma outlines our algorithm. We start with $b_1=(0,1,\dots,p-1,0,\dots,p-1,\dots,0,1,\dots,p-1)$ and cycle over all triples of Davey matrices and for each triple find all solutions $b_2,b_3$ of balanced vectors who pairwise generate the three given Davey matrices. For each solution of $b_1,b_2,b_3$ we find the largest set $\Lambda$ fulfilling the conditions above. And then show that $\Lambda$ is not big enough.

\subsection{Outer loop}
For a given prime $p$ and $1 < n <p$, we start with $b_1$ as in Lemma \ref{3indpvec} and a loop ranging over all Davey matrices of weight $n$. By Lemma \ref{3indpvec} we can first restrict our search to all balanced vectors whose first entry is 0, so we are only interested in Davey matrices $D$ with $D_{0,0}>0$. We will call the set of these matrices $\mathcal{D}_n$. Furthermore we can restrict $b_2$ to those with a second entry of 0 and therefore any Davey matrix $D$ that is raised by the pair ($b_1$,$b_2$) has to fulfill $D_{0,1}>0$. Algorithm \ref{alg::davey_on_vec} now returns one possible $b_2$ such that $D$ is raised by $(b_1,b_2)$. All possible solutions are the same, modulo permutations of those entries in $b_2$ where $b_1$ has the same value on corresponding entries. We can arbitrarily choose one of those solutions, as having $M$ as in Lemma \ref{3indpvec} we immediately see that column operations that do not change $b_1$ result in another matrix $M'$ that also fulfills the properties of the Lemma.

\begin{example}
Let $b_1=(0,1,2,0,1,2)$ a balanced vector in $\Z_3^6$ and $D$ be a Davey matrix: 
$$D= \begin{pmatrix}
    1 & 1 & 0 \\
    1 & 0 & 1 \\
    0 & 1 & 1
  \end{pmatrix}.$$
Then possible $b_2$ are 
$(0,0,1,1,2,2),(0,0,2,1,2,1)$
(since we insist that the first and second entry in $b_2$ is 0). The vectors only differ by transpositions of the third and sixth entry.
\end{example}

\subsection{Inner loops}
Now we have $b_1$ and $b_2$ fix. So the next step is to find all candidates for $b_3$. We know that both $(b_1,b_3)$ and $(b_2,b_3)$ have to give raise to a Davey matrix, so here we have to loop over the set of all Davey matrices in $\mathcal{D}_n$ twice. Now given two Davey matrices $D_1$ and $D_2$ we can find all possible $b_3$ using Algorithm \ref{alg::davey_on_vec_2x2}. Because of the work laid out in Subsection \ref{subsec::search_reduction} we can restrict the set that $D_1$ is chosen from (for fixed $D$ and $D_2$). This is because if $(b_1,b_2)$ give raise to $D$ and we want $(b_2,b_3)$ to give raise to $D_2$, we have some kind of composition of Davey matrices.

\subsection{Testing a triple $(b_1,b_2,b_3)$}
Now in the innermost part of the program we have to test if a triple $(b_1,b_2,b_3)$ is part of a potential solution. By Lemma \ref{3indpvec}  (\ref{lemma::3indpvec::lincomb}) we know that all other rows in a solution $M$ have to be linear combinations of $(b_1,b_2,b_3)$ and from (\ref{lemma::3indpvec::b_0=0}) and (\ref{lemma::3indpvec::balanced}) we know that we are only interested in the linear combinations that result in balanced vectors. This results in a set of valid combinations $V$ as a subset of $\Z_p^3$, i.e.
$$
V=\{(\lambda_1,\lambda_2,\lambda_3) \in \Z_p^3: \lambda_1 b_1+\lambda_2 b_2+\lambda_3 b_3 \text{ is a balanced vector in } \Z_p^{mp}\}
$$ Now to solve the problem we need a set $\Lambda \subset V$ that fulfills the conditions in Lemma \ref{Lemma::win_cond} which can be written as
$$
\forall (\lambda_1,\lambda_2,\lambda_3), (\mu_1,\mu_2,\mu_3) \in \Lambda:
(\lambda_1-\mu_1,\lambda_2-\mu_2,\lambda_3-\mu_3) \in V.
$$
So $\Lambda$ can be seen as a clique in a graph on $V$ with nodes $(v_1,v_2,v_3)$ and $(w_1,w_2,w_3)$ being connected by an edge if $(v_1-w_1,v_2-w_2,v_3-w_3) \in V$. Since the cardinality of $\Lambda$ has to be $np-1$ we need to find a clique of at least that cardinality that contains $(1,0,0),(0,1,0)$ and $(0,0,1)$.

\subsection{Results for $p=5$}
For $p=5$ it turns out that a clique-search is not necessary. Since we have that $(1,0,0),(0,1,0),(0,0,1) \in \Lambda$ we can exclude all nodes not connected to one of those three. The remaining set $R$ has to contain $\Lambda$.
Running the program shows that for each choice of $D$,$D_2$,$D_1$ and $b_3$ the reduced set $R$ has $|R|<np-1$ and therefore cannot contain a $\Lambda$ big enough for a solution. Hence there is no solution to Lemma \ref{3indpvec} and by Theorem \ref{theorem:3DFuglede} we conclude that a subset of $\mathbb{Z}_5^3$ tiles if and only if it is spectral, i.e. the Fuglede conjecture holds for $d=3$ and $p=5$.

\subsection{Remarks on $p \geq 7$}
Similar to the step from $p=3$ to $p=5$, the amount of work required to check all cases increases drastically. So while the program is not restricted to $p=5$ the amount of computations required for the $p=7$ case is far beyond of what is feasible.

\section{Explanation of different parts of the program}

\subsection{The reduction of the search space for $D_1$}
\label{subsec::search_reduction}
\begin{example}
\label{example::search_reduction}
Let 
$$D= \begin{pmatrix}
    1 & 1 & 0 \\
    1 & 0 & 1 \\
    0 & 1 & 1
  \end{pmatrix},
D_2= \begin{pmatrix}
    1 & 0 & 1 \\
    0 & 1 & 1 \\
    1 & 1 & 0
  \end{pmatrix}
  $$
and $b_1, b_2$ such that they give raise to $D$.
Then choosing
$$D_1=\begin{pmatrix}
    2 & 0 & 0 \\
    0 & 0 & 2 \\
    0 & 2 & 0
  \end{pmatrix} \in \mathcal{D}_2$$ doesn't give any valid $b_3$ and can therefore be disregarded.
\end{example}
\begin{proof}
We can show this by observing the pre-image with respect to $D$ and the image with respect to $D_2$ of every 1 in $b_2$. If we think of the vectors $b_1$, $b_2$, $b_3$ written as rows of a $3 \times 6$ matrix, this means finding all ones in $b_2$ and looking at the number above and below. Because of $D$ the numbers above a 1 in $b_2$ can only be one 0 and one 2. And because of $D_2$ the numbers below a 1 in $b_2$ can only be one 1 and one 2. So only looking at $b_1$ and $b_3$ there is either at least on pair (0,1) denoting a 0 on top (in $b_1$) and a 1 at the bottom (at the same position in $b_3$) and (2,2), or we have (0,2) and (2,1). Since $(b_1,b_3)$ give raise to $D_1$ this has to be reflected by the entries with the respective indices. But $(D_1)_{0,1}=0$ which excludes the first case and $(D_1)_{0,2}=0$ which excludes the second case, a contradiction.
\end{proof}
This example shows that we can disregard this choice of $D_1$. Furthermore, here we only followed the 'path' of every 1. We can do the same steps for every number in $\Z_p$ and will get different restrictions most of the time.

\subsection{Algorithms}

The following algorithms refer to the source code contained in Appendix A. The code can also be accessed at \url{ https://doi.org/10.5281/zenodo.2541903}.

\begin{algorithm}[{\ttfamily davey\_on\_vec}] \label{alg::davey_on_vec}
Given a Davey matrix $D$ and a balanced vector $b$ we can generate a corresponding balanced vector $b'$ such that $D$ is raised by $(b,b')$. In general there are several solutions $b'$. $D_{i,j}$ tells us how often a number $i$ lies on top of a number $j$. So for each number in $\Z_p$ we can collect which numbers have to be below it and how often. Then we just traverse $b$ from left to right and assign the numbers according to these collections.
\end{algorithm}

\begin{algorithm}[{\ttfamily davey\_on\_vec\_2x2}]
\label{alg::davey_on_vec_2x2}
This is similar to the previous algorithm, but now we are given two Davey matrices $D^1$ and $D^2$ and two balanced vectors $b^1$ and $b^2$. Our goal is to generate all possible vectors $b^3$ such that the pairs $(b^1,b^3)$ and $(b^2,b^3)$ give raise to $D^1$ and $D^2$ respectively.

To solve this we do a simple tree search. To find the entry $b^3_i=y$ we have to look at the corresponding value $b^1_i=x$, now using $D^1$ we can find the possible values for $y \in \Z_p$ that still allows $(b^1,b^3)$ to give raise to $D^1$, namely we need $D^1_{x,y}>0$. Let $A_1$ be the set of all such values $y$. We can do the same with $b^2$, $D^2$ to get another subset of $\Z_p$: $A_2$ and $b^3_i$ has to be in both sets generated that way. If $A_1 \cap A_2=\emptyset$ we can abandon this branch. If $|A_1 \cap A_2|=1$ we have found $b^3_i$. And if $|A_1 \cap A_2|>1$ then there might be several solutions, so we have to branch and check all possible choices for $b^3_i$.

That way we can step by step find all possible vectors $b^3$.
\end{algorithm}

\begin{algorithm}[{\ttfamily get\_davey\_matrices}]
This algorithm generates all possible Davey matrices of a fixed weight n.

From the note after Definition 5.3. in \cite{Aten2015} we know that a Davey matrix can be written as a sum of permutation matrices. Therefore we need to generate all possible permutation matrices first.

Now any permutation matrix has sum of rows and sum of columns equal to 1. However the sums along the diagonals might not be constant. For each permutation matrix we calculate these sums, a $p$-tuple we will call {\it key} and store matrices with the same key together.

Next we run a loop over all possible combinations of $n$ (where $n$ is the desired weight of the Davey matrices) of those keys. Taking one combination, if (and only if) adding all these keys (as vectors in $\Z^p$) results in $(n,\ldots,n)$ then any sum of the respective matrices to the involved keys will give raise to a Davey matrix!
\end{algorithm}

\begin{algorithm}[{\ttfamily balanced\_linear\_combinations}]
For a given balanced vector $b \in \Z_p^d$ any multiple that isn't divisible by $p$ is a balanced vector again. If however we have three balanced vectors $b_1,b_2,b_3$ determining which linear combinations of the three vectors are again balanced vectors is tricker, so we have to brute-force it by testing all $p^3$ possible combinations.

\end{algorithm}

\begin{algorithm}[{\ttfamily calc\_cache}]
This implements part of the calculation outlined in Subsection \ref{subsec::search_reduction}. Namely the one we can do preemptively, before starting the actual program to generate a cache for quick lookup. Considering Example \ref{example::search_reduction} we see that after choosing 1 as the pivot element what actually mattered were the numbers above: $a=(0,2)$ (in $b_1$) and below (in a potential $b_3$): $b=(1,2)$ and our conditions for the Davey matrix turned out to be that either the entries at $\{(0,1), (2,2)\}$ or the entries at $\{(0,2),(2,1)\}$ were non-zero. These are all possible sets of pairings of numbers in $a$ and $b$, we will call them the keys.
Finding these sets of pairings is the first step of this algorithm, done iteratively over the length of the tuples $a$ and $b$. The second step is for each of these entries to take all pairings and find those Davey matrices that match at least one key, then the set of those matrices will be stored as a result for $(a,b)$.
\end{algorithm}

\begin{algorithm}[{\ttfamily davey\_filtered\_from\_cache}]
This algorithm is the other part of the calculation outlined in Subsection \ref{subsec::search_reduction}, the one done during the main calculation. This algorithm uses an encoding of $D$ in its "pre-images", i.e. a list of tuples, the first tuple consisting of the numbers that result in a 0, next those that result in a 1, and so forth, accounting for multiplicity. $D$ from Example \ref{example::search_reduction} would be encoded as $[(0,1),(0,2),(2,3)]$ (one 0 and one 1 result in a 0) while $D_1$ would be encoded as $[(0,0),(2,2),(1,1)]$.
The given $D_2$ will also be encoded in a similar way by looking at the image for each number, i.e. which numbers in a potential $b_3$ have to be below a 0, a 1, etc. So for $D_2$ as in Example \ref{example::search_reduction} this would be $[(0,2),(1,2),(2,3)]$.
Note that Davey matrices for $p=3$ are always symmetric but this is no longer true for $p \geq 5$.
Now having these encodings we can use the previously calculated {\ttfamily davey\_cache} to get all Davey matrix candidates for what $D_1$ can be.
So we basically do what we did in Example \ref{example::search_reduction}, following the "path" of each value in $\Z_p$. For each number we get a set of Davey matrices that are compatible with that value, so the intersection of those sets is compatible with all of them. That list will be returned.

In Example \ref{example::search_reduction} we only used 1 as the value, but already saw that $D_1$ wasn't in the cache generated for the combination $((0,2),(1,2))$ (second entry of the encoding of $D$ and second entry of the encoding of $D_2$) and therefore not in any intersection.
\end{algorithm}
\pagebreak
\appendix

\section{Source code}
The code can also be accessed at \url{ https://doi.org/10.5281/zenodo.2541903}.
\lstset{
  basicstyle=\ttfamily\footnotesize
}

\begin{lstlisting}[language=Python]
import itertools

from time import clock

import sys

from collections import Counter
from datetime import datetime
from copy import deepcopy


def stacks_from_davey(dmat):
  """
  Takes a Davey matrix dmat and returns
  it in 'stack' format.
  """
  p=len(dmat)
  stacks=[]
  for i in range(p):
    a=[]
    for j in range(p):
      a+=dmat[j][i]*[j]
    stacks.append(a)
  return stacks

def davey_on_vec(dmat,bvec):
  """
  Taking Davey matrix dmat and balanced vector bvec
  returns one possible balanced vector b such that
  dmac is raised by (bvec,b).
  """
  stacks=stacks_from_davey(dmat)
  b=[stacks[x].pop() for x in bvec]
  return b

def davey_on_vec_2x2(dmat1,bvec1,dmat2,bvec2):
  """
  Input: dmat1, bvec1, dmat2, bvec2
  Taking balanced vectors bvec1 and bvec2 as well as
  Davey matrices dmat1 and dmat2, returns all balanced
  vectors b3 such that bvec1, b3 give raise to dmat1
  and bvec2, b3 give raise to dmat2.
  """
  pn=len(bvec1)

  stacks1_=stacks_from_davey(dmat1)
  stacks2_=stacks_from_davey(dmat2)
  unresolved_=range(pn)
  result_=[0]*pn

  out=[]

  queue=[(stacks1_,stacks2_,unresolved_,result_)]

  guess=False
  guessed=False
  while len(queue)>0:
    stacks1,stacks2,unresolved,result=queue.pop(0)
    branch=False
    while len(unresolved)>0:
      resolved=[]
      for i in unresolved:
        s=set(stacks1[bvec1[i]]).intersection(stacks2[bvec2[i]])
        if (len(s)==0):
          unresolved=[]
          resolved=[]
          result=[]
          break
        if (len(s)==1 or branch):
          v=s.pop()
          if branch:
            for w in s:
              result_c=list(result)
              unresolved_c=list(unresolved)
              stacks1_c=list(map(list,stacks1))
              stacks2_c=list(map(list,stacks2))

              unresolved_c.remove(i)
              result_c[i]=w
              stacks1_c[bvec1[i]].remove(w)
              stacks2_c[bvec2[i]].remove(w)
              queue.append((stacks1_c,stacks2_c,unresolved_c,result_c))

          result[i]=v
          stacks1[bvec1[i]].remove(v)
          stacks2[bvec2[i]].remove(v)
          resolved+=[i]
          branch=False
      if len(resolved)>0:
        for j in resolved:
          unresolved.remove(j)
      else:
        #we have to branch somewhere
        branch=True

    if len(result)>0:
      out.append(result)

  return out


def get_davey_matrices(p,n):
  """
  Takes prime p and integer n and returns all Davey matrices mod p
  of weight n.
  """
  perm_dict={}
  for perm in itertools.permutations(range(p)):
    mat=[p*[0] for i in range(p)]
    for i in range(p):
      mat[i][perm[i]]=1
    key=tuple(sum([mat[j][(j+i)%p] for j in range(p)])
        for i in range(p))
    perm_dict[key]=perm_dict.get(key,[])+[mat]

  print "# of different permutation matrices: "+str(len(perm_dict))
  res=set()
  test=p*[n]
  time1=clock()
  for diag_sums in itertools.combinations_with_replacement(
      perm_dict,n):
    if [sum([diag_sums[j][i] for j in range(len(diag_sums))])
        for i in range(p)]==test:
      m=[perm_dict[diag_sums[i]] for i in range(n)]
      for s in itertools.product(*m):
        matrix=[[sum([s[i][j][k] for i in range(len(s))])
            for k in range(p)] for j in range(p)]
        res.add(tuple(map(tuple, matrix)))

  print "# of Davey matrices: "+str(len(res))


  print "time in sec: "+str(clock()-time1)


  return res



def balanced_linear_combinations(b1,b2,b3,p):
  """
  Takes balanced vectors b1, b2 and b3 (mod p) and
  returns a list of triples of numbers (u,v,w)
  such that b1*u+b2*v+b3*w is a balanced vector again.
  """
  test=[(i*p)/len(b1) for i in range(len(b1))]
  ret=[]
  for u,v,w in itertools.product(range(p),repeat=3):
    lcv=[(b1[i]*u+b2[i]*v+b3[i]*w)%p for i in range(len(b1))]
    if (sorted(lcv)==test):
      ret+=[(u,v,w)]
  return ret


def reduced_set(X,p,d=3):
  """
  Input:
  X list of tuples
  p prime
  Returns a sublist of X excluding entries that cannot be
  part of the solution.
  """
  redX=[]
  for e in X:
    add=True
    for i in range(d):
      s=list(e)
      s[i]=(s[i]-1)%p
      if not tuple(s) in X:
        add=False
        break
    if add:
      redX+=[e]
  return redX

def potential_solutions(A,X,size,p):
  mat=[
    [
      1 if (tuple((a[i]-b[i])%p for i in range(len(a))) in X) else 0
      for a in A
    ]
    for b in A
  ]
  rowsum=[sum(mat[i]) for i in range(len(mat))]
  candidates=[A[i] for i in range(len(rowsum)) if rowsum[i]>=size]
  return candidates

def davey_filtered_from_cache(davey_matrix,
    davey_encoded,davey_cache,p):
  davey_filtered=None
  for j in range(p):
    cs=tuple(
        itertools.chain.from_iterable(
        [davey_matrix[i][j]*[i] for i in range(p)]))
    if davey_filtered==None:
      davey_filtered=davey_cache[(davey_encoded[j],cs)].copy()
    else:
      davey_filtered&=davey_cache[(davey_encoded[j],cs)]
  return davey_filtered

def inner_loop(dx,p,n,b1,davey_cache,davey):
  #dimension d is 3, only works for this value
  d=3
  clique_size=p*n-4
  time_step=clock()
  count_reduced_set=Counter()
  count_nr_2x2_davey=Counter()
  most_candidates=0
  most_candidates_info=[]

  m=deepcopy(davey[dx]) #deep copy
  if m[0][1]==0:
    print "############ skip "+str(dx)+" #############"
    return []
  else:
    m[0][1]-=1

  print "+++++++++++++++++++++++++++++++++++++++++++++++"
  print "l = "+str(dx)

  # We can assume that b2[1]=0 from the assumption that b2
  # is lin indep of b1 (b1[1]=1).
  b2=[0,0]+davey_on_vec(m,b1[2:])


  davey_encoded=[tuple(itertools.chain.from_iterable(
      [davey[dx][j][i]*[i] for i in range(p)]))for j in range(p)]

  for dy in range(len(davey)):
    davey_filtered=davey_filtered_from_cache(
        davey[dy],davey_encoded,davey_cache,p)


    for dz in davey_filtered:
      davey_solutions=davey_on_vec_2x2(davey[dz],b1[1:],
          davey[dy],b2[1:])
      count_nr_2x2_davey[len(davey_solutions)]+=1
      for sol in davey_solutions:
        b3=[0]+sol
        valid_combinations=balanced_linear_combinations(b1,b2,b3,p)
        reduced_valid_comb=reduced_set(valid_combinations,p)
        count_reduced_set[len(reduced_valid_comb)]+=1
        if len(reduced_valid_comb)>=clique_size:
            clique_candidates=potential_solutions(
                reduced_valid_comb,valid_combinations,clique_size,p)
            if len(clique_candidates)>most_candidates:
              most_candidates=len(clique_candidates)
              most_candidates_info=[0,dx,dy,dz,b2,b3]
            if (len(clique_candidates)>=clique_size):
              print "#############################################"
              print "###      potential solution found         ###"
              print "#############################################"
              print b1,b2,b3
              print dx,dy,dz
              print clique_candidates
              print "#############################################"


  #end of this davey matrix - report:
  print "time = "+str(datetime.now())
  print "delta time = "+str(clock()-time_step)
  print "Davey 2x2 solutions: "+str(
      sorted(count_nr_2x2_davey.most_common()))
  print "rX sizes: "+str(sorted(count_reduced_set.most_common()))
  print "max row sum = "+str(most_candidates)
  print "max row info: "+str(most_candidates_info)


  return (
    count_nr_2x2_davey,
    count_reduced_set,
    most_candidates,
    most_candidates_info
  )

def calc_cache(p,n,davey_matrices):
  #do precalulations
  key_for_combination={}
  for i in range(p):
    for j in range(p):
      a=p*p*[0]
      a[p*j+i]=1
      key_for_combination[((j,),(i,))]=set([tuple(a)])



  for k in range(2,n+1):
    for t1 in itertools.combinations_with_replacement(range(p),k):
      for t2 in itertools.combinations_with_replacement(range(p),k):
        keys=set()
        for i in range(k):
          a=p*p*[0]
          a[p*t1[0]+t2[i]]=1
          for b in key_for_combination[(t1[1:],t2[:i]+t2[i+1:])]:
            keys.add(tuple([a[j]+b[j] for j in range(p*p)]))
        key_for_combination[(t1,t2)]=keys

  davey_matching={}
  for combination in key_for_combination:
    davey_matching[combination]=set()
    for d_idx in range(len(davey_matrices)):
      for key in key_for_combination[combination]:
        match=True
        for i in range(p*p):
          if davey_matrices[d_idx][i%p][i/p]<key[i]:
            match=False
            break
        if match:
          davey_matching[combination].add(d_idx)
          break
  return davey_matching


def counterexample_d3(p,n,start=0, end=0):
  # dimension d is 3, program only works for this value,
  # so value of d cannot be changed
  d=3

  davey=[]

  #subtract 1 from (0,0) for assumption that first column is 0-vector
  #we can also exclude all matrices with entry (0,0) being equal to 0.
  for dav in get_davey_matrices(p,n):
    e=list(map(list,dav))
    if e[0][0]>0:
      e[0][0]-=1
      davey.append(e)

  #sorting so search can be done in parts.
  davey.sort()
  print len(davey)

  time_start=clock()
  davey_matching=calc_cache(p,n,davey)


  print "cache calc done in "+str(clock()-time_start)
  print len(davey_matching)

  b1=n*range(p)

  bestX=[]
  best=[]
  maxRX=0
  print "==================================================="
  count_nr_2x2_davey_total=Counter()
  count_reduced_set_total=Counter()
  most_candidates_overall=0
  most_candidates_overall_info=[]

  time_start=clock()



  if end==0 or end>len(davey):
    end=len(davey)
  if start==0 or start<0:
    start=0

  for dx in range(start,end):
    #check if matrix should be skipped
    if davey[dx][0][1]==0:
      print "############ skip "+str(dx)+" #############"
    else:
      ret=inner_loop(dx,p,n,b1,davey_matching,davey)
      c2x2_davey,c_red_set,most_cand,most_cand_info=ret
      count_nr_2x2_davey_total+=c2x2_davey
      count_reduced_set_total+=c_red_set
      if most_cand>most_candidates_overall:
        most_candidates_overall=most_cand
        most_candidates_overall_info=most_cand_info


  print "==============================================="
  print "final report"
  print "time elapsed = "+str(clock()-time_start)
  print "Davey 2x2 solutions: "+str(sorted(
      count_nr_2x2_davey_total.most_common()))
  print "rX sizes: "+str(sorted(count_reduced_set_total.most_common()))
  print "max row sum = "+str(most_candidates_overall)
  print "max row info: "+str(most_candidates_overall_info)
  print "==============================================="


def main():
  if len(sys.argv)>=3:
    args=[int(sys.argv[i]) for i in range(1,len(sys.argv))]
    counterexample_d3(*args)
  else:
    sys.stderr.write("need arguments!\n")

  return 0

if __name__ == '__main__':
  main()

\end{lstlisting}

\pagebreak
\bibliography{fuglede_d3p5}
\bibliographystyle{abbrv}

\end{document}